\documentclass[a4paper,12pt]{amsart}
\usepackage{amssymb}
\usepackage{pifont} 
\usepackage{bbding}
\usepackage{wasysym}
\usepackage{ifthen}
\usepackage{cite}
 \usepackage[dvips]{graphicx}
\nonstopmode \numberwithin{equation}{section}
\usepackage{amssymb}
\usepackage{ifthen}
\usepackage{graphicx}
\usepackage{amsmath}
\usepackage[T1]{fontenc} 
\usepackage[utf8]{inputenc}
\usepackage[usenames,dvipsnames]{color}
\usepackage{color}
\usepackage[english]{babel}
\usepackage{fancyhdr}
\usepackage{fancybox}
\usepackage{tikz}

\nonstopmode \numberwithin{equation}{section}
\theoremstyle{plain}

\newtheorem{conj}{Conjecture}

\theoremstyle{definition}

\newtheorem{thm}{Theorem}[section]
\newtheorem{prob}{Problem}[section]
\newtheorem{cor}{Corollary}[section]

\newtheorem{prop}{Proposition}[section]
\newtheorem{rem}{Remark}[section]
\newtheorem{lem}{Lemma}[section]

\newcounter{minutes}
\divide\time by 60
\newcounter{hours}
\multiply\time by 60
\addtocounter{minutes}{-\time}

\newcounter {own}
\def\theown {\thesection       .\arabic{own}}

\newenvironment{pf}[1][]{%
 \vskip 3mm
 \noindent
 \ifthenelse{\equal{#1}{}}%
  {{\slshape Proof. }}%
  {{\slshape #1.} }%
 }%
{\qed\bigskip}

\newcounter{alphabet}

\def\be{\begin{equation}}
\def\ee{\end{equation}}

\newcommand{\bee}{\begin{enumerate}}
\newcommand{\eee}{\end{enumerate}}

\newcommand{\blem}{\begin{lem}}
\newcommand{\elem}{\end{lem}}
\newcommand{\bthm}{\begin{thm}}
\newcommand{\ethm}{\end{thm}}
\newcommand{\bcor}{\begin{cor}}
\newcommand{\ecor}{\end{cor}}
\newcommand{\beg}{\begin{examp}}
\newcommand{\eeg}{\end{examp}}
\newcommand{\begs}{\begin{examples}}
\newcommand{\eegs}{\end{examples}}

\newcommand{\bdefn}{\begin{defn}}
\newcommand{\edefn}{\end{defn}}

\newcommand{\bprob}{\begin{prob}}
\newcommand{\eprob}{\end{prob}}
\newcommand{\bei}{\begin{itemize}}
\newcommand{\eei}{\end{itemize}}

\newcommand{\bcon}{\begin{conj}}
\newcommand{\econ}{\end{conj}}
\newcommand{\bcons}{\begin{conjs}}
\newcommand{\econs}{\end{conjs}}
\newcommand{\bprop}{\begin{prop}}
\newcommand{\eprop}{\end{prop}}
\newcommand{\br}{\begin{rem}}
\newcommand{\er}{\end{rem}}
\newcommand{\brs}{\begin{rems}}
\newcommand{\ers}{\end{rems}}
\newcommand{\bo}{\begin{obser}}
\newcommand{\eo}{\end{obser}}
\newcommand{\bos}{\begin{obsers}}
\newcommand{\eos}{\end{obsers}}
\newcommand{\bpf}{\begin{pf}}
\newcommand{\epf}{\end{pf}}
\newcommand{\ba}{\begin{array}}
\newcommand{\ea}{\end{array}}
\newcommand{\beq}{\begin{eqnarray}}
\newcommand{\beqq}{\begin{eqnarray*}}
\newcommand{\eeq}{\end{eqnarray}}
\newcommand{\eeqq}{\end{eqnarray*}}

\begin{document}

\title{Second Hankel determinant of Logarithmic coefficients of inverse functions in certain classes of univalent functions}
\author{Sanju Mandal}
\address{Sanju Mandal, Department of Mathematics, Jadavpur University, Kolkata-700032, West Bengal, India.}
\email{sanjum.math.rs@jadavpuruniversity.in}

\author{Molla Basir Ahamed$ ^* $}
\address{Molla Basir Ahamed, Department of Mathematics, Jadavpur University, Kolkata-700032, West Bengal, India.}
\email{mbahamed.math@jadavpuruniversity.in}

\subjclass[{AMS} Subject Classification:]{Primary 30A10, 30H05, 30C35, Secondary 30C45}

\def\thefootnote{}
\footnotetext{ {\tiny File:~\jobname.tex,
printed: \number\year-\number\month-\number\day,
          \thehours.\ifnum\theminutes<10{0}\fi\theminutes }
} \makeatletter\def\thefootnote{\@arabic\c@footnote}\makeatother
\begin{abstract} 
The Hankel determinant $H_{2,1}(F_{f^{-1}}/2)$ of logarithmic coefficients is defined as:
\begin{align*}
	H_{2,1}(F_{f^{-1}}/2):= \begin{vmatrix}
		\Gamma_1 & \Gamma_2 \\
		\Gamma_2 & \Gamma_3
	\end{vmatrix}=\Gamma_1\Gamma_3-\Gamma^2_2,
\end{align*}
where $\Gamma_1, \Gamma_2,$ and $\Gamma_3$ are the first, second and third logarithmic coefficients of inverse functions belonging to the class $\mathcal{S}$ of normalized univalent functions. In this article, we establish sharp inequalities $|H_{2,1}(F_{f^{-1}}/2)|\leq 19/288$, $|H_{2,1}(F_{f^{-1}}/2)| \leq 1/144$, and $|H_{2,1}(F_{f^{-1}}/2)| \leq 1/36$ for the logarithmic coefficients of inverse functions, considering starlike and convex functions, as well as functions with bounded turning of order $1/2$, respectively.
\end{abstract}
\maketitle
\pagestyle{myheadings}
\markboth{S. Mandal and M. B. Ahamed}{Second Hankel determinant of Logarithmic coefficients of inverse functions}
\tableofcontents
\section{Introduction}

The coefficient problem is a fundamental aspect of geometric functions theory, and finding sharp results for this problem is a central objective because it offers crucial insights into the behavior of functions on geometric spaces and has far-reaching implications in various mathematical and scientific disciplines. Let $\mathcal{H}$ be the class of functions $ f $ which are holomorphic in the open unit disk $\mathbb{D}=\{z\in\mathbb{C}: |z|<1\}$. Then $\mathcal{H}$ is a locally convex topological vector space endowed with the topology of uniform convergence over compact subsets of $\mathbb{D}$. Let $\mathcal{A}$ denote the class of functions $f\in\mathcal{H}$ such that $f(0)=0$ and $f^{\prime}(0)=1$. That is, the function $f$ of the form
\begin{align}\label{eq-1.1}
	f(z)=z+ \sum_{n=2}^{\infty}a_nz^n,\; \mbox{for}\; z\in\mathbb{D}.
\end{align} 
Let $\mathcal{S}$ denote the subclass of all functions in $\mathcal{A}$ which are univalent. For a general theory of univalent functions and their significance in coefficients problem, we refer consulting the books \cite{Duren-1983-NY,Goodman-1983}. For $f\in\mathcal{S}$, we define
\begin{align}\label{eq-1.2}
	F_{f}(z):=\log\dfrac{f(z)}{z}=2\sum_{n=1}^{\infty}\gamma_{n}(f)z^n, \;\; z\in\mathbb{D},\;\;\log 1:=0,
\end{align}
a logarithmic function associated with $f\in\mathcal{S}$. The numbers $\gamma_{n}:=\gamma_{n}(f)$ ($ n\in\mathbb{N} $) are called the logarithmic coefficients of $f$. Although the logarithmic coefficients $\gamma_{n}$ play a critical role in the theory of univalent functions, it appears that only a limited number of sharp bounds have been established for them. As is well known, the logarithmic coefficients play a crucial role in Milin’s conjecture (\cite{Milin-1977-ET}, see also \cite[p.155]{Duren-1983-NY}).
Milin conjectured that for $f\in\mathcal{S}$ and $n\geq2$,
\begin{align*}
	\sum_{m=1}^{n}\sum_{k=1}^{m}\left(k|\gamma_{k}|^2 -\frac{1}{k}\right)\leq 0
\end{align*}
where the equality holds if, and only if, $f$ is a rotation of the Koebe function. De Branges \cite{Branges-AM-1985} proved Milin conjecture which confirmed the famous Bieberbach conjecture. On the other hand, one of reasons for more attention has been given to the Logarithmic coefficients is that the sharp bound for the class $\mathcal{S}$ is known only for $\gamma_{1}$ and $\gamma_{2}$, namely
\begin{align*}
	|\gamma_{1}|\leq 1, \;\; |\gamma_{2}|\leq \dfrac{1}{2}+ \dfrac{1}{e} =0.635\ldots
\end{align*}
It is still an open problem to find the sharp bounds of $\gamma_{n}$, $n\geq 3$, for the class $\mathcal{S}$. Estimating the modulus of logarithmic coefficients for $f\in\mathcal{S}$ and various sub-classes has been considered recently by several authors. For more information on the topic, we recommend consulting the articles \cite{Ali-Allu-PAMS-2018, Ali-Allu-Thomas-CRMCS-2018,Cho-Kowalczyk-kwon-Lecko-Sim-RACSAM-2020,Girela-AASF-2000,Thomas-PAMS-2016} and references therein.\vspace{2mm}

Let $F$ be the inverse function of $f\in\mathcal{S}$ defined in a neighborhood of the origin with the Taylor series expansion
\begin{align}\label{eq-1.3}
	F(w):=f^{-1}(w)= w+\sum_{n=2}^{\infty} A_n w^n,
\end{align}
where we may choose $|w|<1/4$, as we know that the famous K$\ddot{\mbox{o}}$ebe’s $1/4$-theorem ensures that, for each univalent function $f$ defined in $\mathbb{D}$, it inverse $f^{-1}$ exists at least on a disc of radius $1/4$. Using a variational method, L$\ddot{\mbox{o}}$wner \cite{Lowner-IMA-1923} obtained the sharp estimate $ |A_n|\leq K_n \;\mbox{for each}\; n, $
where $K_n=(2n)!/(n!(n+1)!)$ and $K(w)= w +K_2 w^2 +K_3 w^3 +\cdots$ is the inverse of the K\"oebe function. There has been a good deal of interest in determining the behavior of the inverse coefficients of $f$ given in \eqref{eq-1.1} when the corresponding function $f$ is restricted to some proper geometric subclasses of $\mathcal{S}$.\vspace{2mm}

Let $f(z)=z+ \sum_{n=2}^{\infty}a_nz^n$ be a function in class $\mathcal{S}$. Science $f(f^{-1}(w))=w$ and using \eqref{eq-1.3} we obtain
\begin{align}\label{eq-1.4}
   \begin{cases}
	A_2= -a_2, \\ A_3=-a_3 +2a^2_{2}, \\ A_4=- a_4 +5a_2 a_3 -5a^3_{2}, \\ A_5=- a_5+6a_4 a_2- 21a_3 a^2_{2} +3a^2_{3} +14a^4_{2}.
   \end{cases}
\end{align}
The notation of the logarithmic coefficient of inverse of $f$ was introduced by Ponnusamy \textit{et al.} (see \cite{Ponnusamy-Sharma-Wirths-RM-2018}). As with $f$, the logarithmic inverse coefficients $\Gamma_n$, $n\in\mathbb{N}$, of $F$ are defined by the equation
\begin{align}\label{eq-1.5}
	\log\left(\frac{F(w)}{w}\right)=2\sum_{n=1}^{\infty} \Gamma_n(F) w^n \;\;\;\; \mbox{for} \;\;|w|<1/4.
\end{align}
The author's obtained the sharp bound for the logarithmic inverse coefficients for the class $\mathcal{S}$. In fact, Ponnusamy \textit{et al.} \cite{Ponnusamy-Sharma-Wirths-RM-2018} established for $f\in\mathcal{S}$ that
\begin{align*}
	|\Gamma_n(F)|\leq\frac{1}{2n}\binom{2n}{n}
\end{align*}
and showed that the equality holds only for the K\"oebe function or  its rotations.
Differentiating \eqref{eq-1.5} and using \eqref{eq-1.4}, we obtain
\begin{align}\label{eq-1.6}
	\begin{cases}
		\Gamma_1=-\frac{1}{2}a_2, \\ \Gamma_2=-\frac{1}{2}\left(a_3 -\frac{3}{2}a^2_{2}\right), \\ \Gamma_3=-\frac{1}{2}\left(a_4 -4a_2 a_3 +\frac{10}{3}a^3_{2}\right), \\ \Gamma_4 =-\frac{1}{2} \left(a_5 -5a_4 a_2 +15 a_3 a^2_{2} -\frac{5}{2}a^2_{3} -\frac{35}{4}a^4_{2}\right).
	\end{cases}
\end{align}
The evaluation of Hankel determinants has been a major concern in geometric function theory, where these determinants are formed by employing coefficients of analytic functions $f$ that are characterized by \eqref{eq-1.1} and defined within the region $\mathbb{D}$. Hankel matrices (and determinants) have emerged as fundamental elements in different areas of mathematics, finding a wide array of applications \cite{Ye-Lim-FCM-2016}. The primary objective of this study is to determine the sharp bound for the second Hankel determinant, which involves the use of logarithmic coefficients. To begin, we present the definitions of Hankel determinants in situations where $f\in \mathcal{A}$.\vspace{1.2mm}

The Hankel determinant $H_{q,n}(f)$ of Taylor's coefficients of functions $f\in\mathcal{A}$ represented by \eqref{eq-1.1} is defined for $q,n\in\mathbb{N}$ as follows:
\begin{align*}
	H_{q,n}(f):=\begin{vmatrix}
		a_{n} & a_{n+1} &\cdots& a_{n+q-1}\\ a_{n+1} & a_{n+2} &\cdots& a_{n+q} \\ \vdots & \vdots & \vdots & \vdots \\ a_{n+q-1} & a_{n+q} &\cdots& a_{n+2(q-1)}
	\end{vmatrix}.
\end{align*}
Recently, Kowalczyk and Lecko \cite{Kowalczyk-Lecko-BAMS-2022} proposed a Hankel determinant whose elements are the logarithmic coefficients of $f\in\mathcal{S}$, realizing the extensive use of these coefficients. Inspired by these ideas, we introduce the investigation of the Hankel determinant $H_{q,n}(F_{f^{-1}}/2)$, wherein the elements are logarithmic coefficients of inverse functions of $f^{-1}\in\mathcal{S}$. The determinant $H_{q,n}(F_{f^{-1}}/2)$ is expressed as follows:
\begin{align*}
	H_{q,n}(F_{f^{-1}}/2)=\begin{vmatrix}
		\Gamma_{n} & \Gamma_{n+1} &\cdots& \Gamma_{n+q-1}\\ \Gamma_{n+1} & \Gamma_{n+2} &\cdots& \Gamma_{n+q} \\ \vdots & \vdots & \vdots & \vdots \\ \Gamma_{n+q-1} & \Gamma_{n+q} &\cdots& \Gamma_{n+2(q-1)}
	\end{vmatrix}.
\end{align*}

The extensive exploration of sharp bounds of the Hankel determinants for starlike, convex, and other function classes has been undertaken in various studies (see \cite{Ponnusamy-Sugawa-BDSM-2021, Kowalczyk-Lecko-RACSAM-2023, Raza-Riza-Thomas-BAMS-2023, Sim-Lecko-Thomas-AMPA-2021, Kowalczyk-Lecko-BAMS-2022}), and their precise bounds have been successfully established. Recently, there has been a growing interest in studying Hankel determinants incorporating logarithmic coefficients within certain sub-classes of starlike, convex, univalent, strongly starlike, and strongly convex functions (see \cite{Allu-Arora-Shaji-MJM-2023, Kowalczyk-Lecko-BAMS-2022, Kowalczyk-Lecko-BMMS-2022, Krishna-Reddy-CVEE-2015, Sim-Thomas-CVEE-2022} and the relevant literature). Despite this, the sharp bound of Hankel determinants of logarithmic coefficients remains relatively unknown, prompting the need for comprehensive studies for various function classes.\vspace{2mm} 


Let $\alpha\in[0,1)$. A function $f\in\mathcal{A}$ is called starlike of order $\alpha$ if
\begin{align}\label{eq-1.7}
	\mbox{Re}\left(\frac{zf^{\prime}(z)}{f(z)}\right)>\alpha,\;\;z\in\mathbb{D}.
\end{align}

A function $f\in\mathcal{A}$ is called convex of order $\alpha$ if
\begin{align}\label{eq-1.8}
	\mbox{Re}\left(1+\frac{zf^{\prime\prime}(z)}{f^{\prime}(z)}\right) >\alpha,\;\;z\in\mathbb{D}.
\end{align}
A function $f\in\mathcal{A}$ is called bounded turning functions of order $\alpha$ if
\begin{align}\label{eq-1.9}
	\mbox{Re}f^{\prime}(z)>\alpha,\;\;z\in\mathbb{D}.
\end{align}
These classes are usually denoted as $\mathcal{S}^{*}(\alpha)$, $\mathcal{S}^{c}(\alpha)$ and $\mathcal{R}(\alpha)$ respectively. The classes $\mathcal{S}^{*}(0):=\mathcal{S}^{*}$, $\mathcal{S}^{c}(0)=:\mathcal{S}^{c}$ and $\mathcal{R}(0):=\mathcal{R}$ consist of starlike, convex and bounded turning respectively. Both classes $\mathcal{S}^{*}(\alpha)$ and $\mathcal{S}^{c}(\alpha)$ were introduced by Robertson \cite{Robertson-AM-1936}(e.g,\cite[Vol. I,p. 138]{Goodman-1983}).
An important role is played by the class $\mathcal{S}^{*}(1/2)$. One of
the significant results belongs to Marx \cite{Marx-MA-1932} and Strohh\"acker
\cite{Strohhäcker-MZ-1933}. They proved that $\mathcal{S}^{c}\subset \mathcal{S}^{*}(1/2)$ (see also \cite[Theorem 2.6a, p. 57] {Miller-Mocanu-2000}).\vspace{2mm}

In light of the significance of logarithmic coefficients in problems pertaining to geometric function theory, there has been a growing interest in computing the Hankel determinant using logarithmic coefficients in recent years. But, sharp bound of Hankel determinant of logarithmic coefficients of inverse function are known very little. We considering the second Hankel determinant of $F_{f^{-1}}/2$ is defined as
\begin{align}\label{eq-1.10}
	H_{2,1}(F_{f^{-1}}/2)=\Gamma_{1}\Gamma_{3} -\Gamma^2_{2}=\frac{1}{48} \left(13a^4_2 -12a^2_2 a_3 - 12 a^2_3 + 12 a_2 a_4\right).
\end{align}
The objective of this study is to investigate the sharp bounds of the Hankel determinant $H_{2,1}(F_{f^{-1}}/2)$ for classes of functions, including starlike, convex, and bounded turning functions of order $1/2$. \vspace{2mm}

\section{Preliminary results}
The Carath$\acute{e}$odory class $\mathcal{P}$ and its coefficients bounds plays a significant roles in establishing the bounds of Hankel determinants. The class $\mathcal{P}$ of analytic functions $h$ defined for $z\in\mathbb{D}$ is given by
\begin{align}\label{eq-2.1}
	p(z)=1+\sum_{n=1}^{\infty}c_n z^n
\end{align}
with positive real part in $\mathbb{D}$. A member of $\mathcal{P}$ is called a Carath$\acute{e}$odory function. It is known that $c_n\leq 2$, $n\geq 1$ for a function $p\in\mathcal{P}$ (see \cite{Duren-1983-NY}).\vspace{2mm}

In this section, we provide crucial lemmas that will be utilized to establish the main results of this paper. Parametric representations of the coefficients are often useful in finding the bound of Hankel determinants, in this regard, Libera and Zlotkiewicz \cite{Libera-Zlotkiewicz-PAMS-1982, Libera-Zlotkiewicz-PAMS-1983} derived the following parameterizations of possible values of $c_2$ and $c_3$.
\begin{lem}\cite{Libera-Zlotkiewicz-PAMS-1982,Libera-Zlotkiewicz-PAMS-1983}\label{lem-2.1}
If $p\in\mathcal{P}$ is of the form \eqref{eq-2.1} with $c_1\geq 0$, then 
\begin{align}\label{eq-2.2}
	&c_1=2\tau_1,\\\label{eq-2.3} &c_2=2\tau^2_1 +2(1-\tau^2_1)\tau_2
\end{align}
and
\begin{align}\label{eq-2.4}
	c_3= 2\tau^3_1  + 4(1 -\tau^2_1)\tau_1\tau_2 - 2(1 - \tau^2_1)\tau_1\tau^2_2 + 
	2(1 - \tau^2_1)(1 - |\tau_2|^2)\tau_3
\end{align}
for some $\tau_1\in[0,1]$ and $\tau_2,\tau_3\in\overline{\mathbb{D}}:= \{z\in\mathbb{C}:|z|\leq 1\}$.\vspace{1.2mm}

For $\tau_1\in\mathbb{T}:=\{z\in\mathbb{C}:|z|=1\}$, there is a unique function $p\in\mathcal{P}$ with $c_1$ as in \eqref{eq-2.2}, namely
\begin{align*}
	p(z)=\frac{1+\tau_1 z}{1-\tau_1 z}, \;\;z\in\mathbb{D}.
\end{align*}

For $\tau_1\in\mathbb{D}$ and $\tau_2\in\mathbb{T}$, there is a unique function $p\in\mathcal{P}$ with $c_1$ and $c_2$ as in \eqref{eq-2.2} and \eqref{eq-2.3}, namely
\begin{align*}
	p(z)=\frac{1+(\overline{\tau_1}\tau_2 +\tau_1)z+\tau_2 z^2}{1 +(\overline{\tau_1}\tau_2 -\tau_1)z-\tau_2 z^2}, \;\;z\in\mathbb{D}.
\end{align*}

For $\tau_1,\tau_2\in\mathbb{D}$ and $\tau_3\in\mathbb{T}$, there is a unique function $p\in\mathcal{P}$ with $c_1,c_2$ and $c_3$ as in \eqref{eq-2.2}--\eqref{eq-2.4}, namely
\begin{align*}
	p(z)=\frac{1+(\overline{\tau_2}\tau_3+\overline{\tau_1}\tau_2 +\tau_1)z +(\overline{\tau_1}\tau_3+ \tau_1\overline{\tau_2}\tau_3 +\tau_2)z^2 +\tau_3 z^3}{1 +(\overline{\tau_2}\tau_3+ \overline{\tau_1}\tau_2 -\tau_1)z +(\overline{\tau_1}\tau_3- \tau_1\overline{\tau_2}\tau_3 -\tau_2)z^2 -\tau_3 z^3}, \;\;z\in\mathbb{D}.
\end{align*}
\end{lem}

\begin{lem}\cite{Cho-Kim-Sugawa-JMSJ-2007}\label{lem-2.2}
Let $A, B, C$ be real numbers and
\begin{align*}
	Y(A,B,C):=\max\{|A+ Bz +Cz^2| +1-|z|^2: z\in\overline{\mathbb{D}}\}.
\end{align*}
\noindent{(i)} If $AC\geq 0$, then
\begin{align*}
	Y(A,B,C)=\begin{cases}
		|A|+|B|+|C|, \;\;\;\;\;\;\;\;\;\;\;\;\;|B|\geq 2(1-|C|), \vspace{2mm}\\ 1+|A|+\dfrac{B^2}{4(1-|C|)}, \;\;\;\;\;|B|< 2(1-|C|).
	\end{cases}
\end{align*}
\noindent{(ii)} If $AC<0$, then
\begin{align*}
	Y(A,B,C)=\begin{cases}
		1-|A|+\dfrac{B^2}{4(1-|C|)}, \;\;\;\;-4AC(C^{-2}-1)\leq B^2\land|B|< 2(1-|C|),\vspace{2mm} \\ 1+|A|+\dfrac{B^2}{4(1+|C|)}, \;\;\;\; B^2<\min\{4(1+|C|)^2,-4AC(C^{-2}-1)\}, \vspace{2mm} \\ R(A,B,C), \;\;\;\;\;\;\;\;\;\;\;\;\;\;\;\;\;\;\; otherwise,
	\end{cases}
\end{align*}
where
\begin{align*}
	R(A,B,C):= \begin{cases}
		|A|+|B|-|C|, \;\;\;\;\;\;\;\;\;\;\;\;\; |C|(|B|+4|A|)\leq |AB|, \vspace{2mm}\\ -|A|+|B|+|C|, \;\;\;\;\;\;\;\;\;\;\; |AB|\leq |C|(|B|-4|A|), \vspace{2mm}\\ (|C| +|A|)\sqrt{1-\dfrac{B^2}{4AC}}, \;\;\; otherwise.
	\end{cases}
\end{align*}
\end{lem}
To ensure a clear presentation, we have divided the content of Hankel determinants of logarithmic coefficients of inverse functions into three sections, each focusing on different classes of functions belonging to $\mathcal{A}$. Our main results for starlike, convex, and bounded turning functions are demonstrated separately within these sections.
\section{Sharp bound of $ |H_{2,1}(F_{f^{-1}}/2)| $ for the class $\mathcal{S}^{*}(1/2)$}
We obtain the following result finding the sharp bound of $ |H_{2,1}(F_{f^{-1}}/2)| $ for functions in the class $\mathcal{S}^{*}(1/2)$.
\begin{thm}\label{th-3.1}
Let $ f\in\mathcal{S}^{*}(1/2) $. Then
\begin{align}\label{eq-3.1}
	|H_{2,1}(F_{f^{-1}}/2)|\leq\frac{19}{288}.
\end{align}
The inequality is sharp.
\end{thm}
\begin{proof}
Let $f\in\mathcal{S}^{*}(1/2)$. Then in view of \eqref{eq-1.1} and \eqref{eq-1.7}, it follows that
\begin{align}\label{eq-3.2}
	\frac{zf^{\prime}(z)}{f(z)}=\frac{1}{2}(p(z) +1), \;\; z\in\mathbb{D}, 
\end{align} 
for some $p\in\mathcal{P}$ of the form \eqref{eq-2.1}. Since the class $\mathcal{P}$ and $|H_{2,1}(F_{f^{-1}}/2)|$ is invariant under rotation, we may assume that $c_1\in[0,2]$(see \cite{Carathéodory-MA-1907}; see also \cite[Vol. I, p.80, Theorem 3]{Goodman-1983}), that is, in view of \eqref{eq-2.1}, that $\tau_1\in[0,1]$. Substituting \eqref{eq-1.1} and \eqref{eq-2.1} into \eqref{eq-3.2} and equating coefficients, we obtain
\begin{align}\label{eq-3.3}
		a_2=\dfrac{1}{2}c_1,\;a_3=\dfrac{1}{8}(2c_2 + c^2_{1}), \; \mbox{and}\;a_4=\dfrac{1}{48}(8 c_3 + 6c_1 c_2 +c^3_{1}).
\end{align}
Using \eqref{eq-1.6} and \eqref{eq-3.3}, an computation leads to
\begin{align}\label{eq-3.4}
	H_{2,1}(F_{f^{-1}}/2)&=\frac{1}{48} \left(13a^4_2 -12a^2_2 a_3 - 12 a^2_3 + 12 a_2 a_4\right)\\&\nonumber=\frac{1}{384}\left(3c^4_{1} -6c^2_{1} c_2 -6c^2_{2} +8c_1 c_3\right).
\end{align}
By the Lemma \ref{lem-2.1} and \eqref{eq-3.4}, a simple computation shows that
\begin{align}\label{eq-3.5}
	H_{2,1}(F_{f^{-1}}/2)&=\frac{1}{48}\left(\tau^4_{1} -4(1-\tau^2_{1})\tau^2_{1}\tau_{2}-(1-\tau^2_{1})(3+\tau^2_{1})\tau^2_{2} \right.\\&\nonumber\left.\quad +4\tau_{1}\tau_{3}(1-\tau^2_{1})(1-|\tau^2_{2}|)\right).
\end{align}
Below, we explore the following possible cases involving $\tau_{1}$: \vspace{2mm}

\noindent{\bf Case 1.} Suppose that $\tau_{1}=1$. Then from \eqref{eq-3.5}, we easily obtain 
\begin{align*}
	|H_{2,1}(F_{f^{-1}}/2)|=\frac{1}{48}.
\end{align*}

\noindent{\bf Case 2.} Let $\tau_{1}=0$. Then from \eqref{eq-3.5}, we see that 
\begin{align*}
	|H_{2,1}(F_{f^{-1}}/2)|=\frac{1}{16}|\tau_{2}|^2\leq\frac{1}{16}.
\end{align*}
\noindent{\bf Case 3.} Suppose that $\tau_{1}\in(0,1)$. Applying the triangle inequality in \eqref{eq-3.5} and by using the fact that $|\tau_{3}|\leq 1$, we obtain
\begin{align}\label{eq-3.6}
	\nonumber|H_{2,1}(F_{f^{-1}}/2)|&\leq\frac{1}{48}\left(|\tau^4_{1} -4(1-\tau^2_{1})\tau^2_{1}\tau_{2}-(1-\tau^2_{1})(3+\tau^2_{1})\tau^2_{2}| \right.\\&\nonumber\left.\quad +4\tau_{1}(1-\tau^2_{1}) (1-|\tau^2_{2}|) \right)\\&=\frac{1}{12}\tau_{1}(1-\tau^2_{1}) \left(|A+B\tau_{2}+C\tau^2_{2}| +1 -|\tau_{2}|^2\right),
\end{align}
where 
\begin{align*}
	A:=\frac{\tau^3_{1}}{4(1-\tau^2_{1})}, \;\; B:=-\tau_{1}
	\;\; \mbox{and}\;\; C:=\frac{-(3+\tau^2_{1})}{4\tau_{1}}.
\end{align*}
Observe that $AC< 0$. Hence, we can apply case (ii) of Lemma \ref{lem-2.2}. Next, we check all the conditions of case (ii). \vspace{2mm}

\noindent{\bf 3(a)} We note the inequality
\begin{align*}
	-4AC\left(\frac{1}{C^2} -1\right)- B^2=\frac{(3+\tau^2_{1})\tau^2_{1}}{4(1- \tau^2_{1})}\left(\frac{16\tau^2_{1}}{(3+\tau^2_{1})^2} - 1\right) -\tau^2_{1}\leq 0
\end{align*}
is equivalent to
\begin{align*}
	\frac{(3+ 2\tau^2_{1} -5\tau^4_{1})}{4(1-\tau^2_{1})(3+\tau^2_{1})} \leq 0
\end{align*}
which is evidently  holds for $\tau_{1}\in(0,1)$. However, the inequality $|B|<2(1-|C|)$ is equivalent to $3\tau^2_{1} -4\tau_{1} + 3<0$, which is false for $\tau_{1}\in(0,1)$.

\noindent{\bf 3(b)} Since
\begin{align*}
	4(1+|C|)^2 =\frac{(3+4\tau_{1}+\tau^2_{1})^2}{4\tau^2_{1}}>0
\end{align*}
and 
\begin{align*}
	-4AC\left(\frac{1}{C^2} -1\right)=-\frac{(3+\tau^2_{1})\tau^2_{1}(9-10\tau^2_{1} +\tau^4_{1})}{4(1-\tau^2_{1})(3+\tau^2_{1})^2}<0,
\end{align*}
a simple computation shows that the inequality
\begin{align*}
	\tau^2_{1}=B^2<\min\left\{4(1+|C|)^2,-4AC\left(\frac{1}{C^2} -1\right)\right\}=-\frac{(3+\tau^2_{1})\tau^2_{1}(9-10\tau^2_{1} +\tau^4_{1})}{4(1-\tau^2_{1})(3+\tau^2_{1})^2}
\end{align*}
is false for $\tau_{1}\in(0,1)$.\vspace{1.2mm}

\noindent{\bf 3(c)} Next note that the inequality
\begin{align*}
	|C|(|B| +4|A|) -|AB|=\frac{(3+\tau^2_{1})}{4\tau_{1}}\left(\tau_{1} +\frac{\tau^3_{1}}{(1-\tau^2_{1})}\right)-\frac{\tau^4_{1}}{4(1-\tau^2_{1})}\leq 0
\end{align*}
is equivalent to $3+\tau^2_{1} -\tau^4_{1}\leq 0$, which is false for $\tau_{1}\in(0,1)$.

\noindent{\bf 3(d)} Note that the inequality
\begin{align*}
	|AB|-|C|(|B|-4|A|)= \frac{\tau^4_{1}}{4(1-\tau^2_{1})}- \frac{(3+\tau^2_{1})}{4\tau_{1}}\left(\tau_{1} -\frac{\tau^3_{1}}{(1-\tau^2_{1})}\right)\leq 0
\end{align*}
is equivalent to $3\tau^4_{1} +5\tau^2_{1} -3\leq 0$, which is true for $0<\tau_{1}\leq \tau^{''}_{1}=\sqrt{\frac{1}{6}(-5+\sqrt{61})}\approx 0.684379$. Applying Lemma \ref{lem-2.2} for $0<\tau_{1}\leq \tau^{''}_{1}$,
we obtain
\begin{align}\label{eq-3.7}
	\nonumber|H_{2,1}(F_{f^{-1}}/2)|&\leq \frac{1}{12}\tau_{1}(1-\tau^2_{1}) (-|A|+|B|+|C|) \\&=\frac{1}{48}(3+2\tau^2_{1} -6\tau^4_{1})=\phi(\tau_{1}),
\end{align}
where 
\begin{align*}
	\phi(t):=\frac{1}{48}(3+2t^2 -6t^4).
\end{align*}
Since $\phi^{\prime}(t)=0$ for $t\in(0,1)$ holds only for $t_0=1/\sqrt{6}<\tau^{''}_{1}$, we deduce that $\phi$ is increasing in $[0,t_0]$ and decreasing in $[t_0,\tau^{''}_{1}]$. Therefore, in  $0<\tau_{1}\leq \tau^{''}_{1}$, we obtain
\begin{align*}
	\phi(\tau_{1})\leq\phi(t_0)=\frac{19}{288}\approx 0.0659722.
\end{align*}
Hence, from \eqref{eq-3.7} we see that
\begin{align*}
	|H_{2,1}(F_{f^{-1}}/2)|\leq\phi(\tau_{1})\leq\phi(t_0)=\frac{19}{288}.
\end{align*}
\noindent{\bf 3(e)} Applying Lemma \ref{lem-2.2} for $\tau^{''}_{1}<\tau_{1}<1$, we get
\begin{align}\label{eq-3.8}
	\nonumber|H_{2,1}(F_{f^{-1}}/2)|&\leq\frac{1}{12}\tau_{1}(1-\tau^2_{1}) (|C| +|A|)\sqrt{1-\dfrac{B^2}{4AC}} \\&=\frac{(3-2\tau^2_{1})}{48}\sqrt{\frac{7-3\tau^2_{1}}{3+\tau^2_{1}}}=\psi(\tau_{1}),
\end{align}
where
\begin{align*}
	\psi(t):=\frac{(3-2t^2)}{48}\sqrt{\frac{7-3t^2}{3+t^2}}.
\end{align*}
A simple computation shows that
\begin{align*}
	\psi^{\prime}(t)=\frac{t(-33 + 10t^2 + 3t^4)}{12(t^2 +3)^{\frac{3}{2}}\sqrt{7-3t^2}}< 0, \;\;\;\; \tau^{''}_{1}<t <1,
\end{align*}
hence, the function $\psi$ is decreasing. Therefore, we have $\psi(t)\leq\psi(\tau^{''}_{1})$ for $\tau^{''}_{1}\leq t \leq 1$. Hence, from \eqref{eq-3.8}, we obtain that
\begin{align*}
	|H_{2,1}(F_{f^{-1}}/2)|\leq\psi(\tau_{1})\leq\psi(\tau^{''}_{1})\approx 0.0545938.
\end{align*}
Summarizing Cases $1$, $2$, and $3$, the inequality \eqref{eq-3.1} is established.\vspace{2mm}

The proof can be concluded by establishing the sharpness of the bound. In order to show that, we consider the function $f\in\mathcal{A}$ as follows
\begin{align*}
	\frac{zf^{\prime}(z)}{f(z)}=\frac{1}{1 -\beta z^2}, \;\; z\in\mathbb{D}, 
\end{align*}
where $\beta= {\sqrt{19}}/{3\sqrt{2}}$, for which $a_2=a_4=0$ and $a_3=\beta/2$. By a simple computation, it can be easily shown that $|H_{2,1}(F_{f^{-1}}/2)|=19/288$, and this shows that the bound in the result is sharp. This completes the proof.
\end{proof}
\section{Sharp bound of $ |H_{2,1}(F_{f^{-1}}/2)| $ for the class $\mathcal{S}^{c}(1/2)$}
We obtain the following sharp bound of $ |H_{2,1}(F_{f^{-1}}/2)| $ for functions in the class $\mathcal{S}^{c}(1/2)$. 
\begin{thm}\label{th-2.3}
Let $ f\in\mathcal{S}^{c}(1/2) $. Then
\begin{align}\label{eq-4.1}
	|H_{2,1}(F_{f^{-1}}/2)|\leq\frac{1}{144}.
\end{align}
The inequality is sharp.
\end{thm}
\begin{proof}
Let $f\in\mathcal{S}^{c}(1/2)$. Then in view of \eqref{eq-1.1} and \eqref{eq-1.8}, we have
\begin{align}\label{eq-4.2}
	1+ \frac{zf^{\prime\prime}(z)}{f^{\prime}(z)}=\frac{1}{2}(p(z) +1), \;\; z\in\mathbb{D}, 
\end{align} 
for some $p\in\mathcal{P}$ of form \eqref{eq-2.1}. By the similar argument that being used previously, we see that $\tau_1\in[0,1]$. Substituting \eqref{eq-1.1} and \eqref{eq-2.1} into \eqref{eq-4.2} and equating the coefficients, we obtain
\begin{align}\label{eq-4.3}
		a_2=\dfrac{1}{4}c_1,\; a_3=\dfrac{1}{24}(2c_2 + c^2_{1}),\; \mbox{and}\;a_4=\dfrac{1}{192}(8 c_3 + 6c_1 c_2 +c^3_{1}).
\end{align}
 Using \eqref{eq-1.6} and \eqref{eq-4.3}, it is easy to see that
\begin{align}\label{eq-4.4}
	H_{2,1}(F_{f^{-1}}/2)&=\frac{1}{48} \left(13a^4_2 -12a^2_2 a_3 - 12 a^2_3 + 12 a_2 a_4\right)\\&\nonumber=\frac{1}{36864}\left(11c^4_{1} -40c^2_{1} c_2 -64c^2_{2} +96c_1 c_3\right).
\end{align}
By the Lemma \ref{lem-2.1} and \eqref{eq-4.4}, a straightforward computation shows that
\begin{align}\label{eq-4.5}
	H_{2,1}(F_{f^{-1}}/2)&=\frac{1}{2304}\left(-\tau^4_{1} -4(1-\tau^2_{1})\tau^2_{1}\tau_{2}-8(1-\tau^2_{1})(2+\tau^2_{1})\tau^2_{2} \right.\\&\nonumber\left.\quad +24\tau_{1}\tau_{3}(1-\tau^2_{1})(1-|\tau^2_{2}|)\right).
\end{align}
Below, we discuss the following possible cases on $\tau_{1}$: \vspace{2mm}

\noindent{\bf Case 1.} Suppose that $\tau_{1}=1$. Then from \eqref{eq-4.5}, we easily obtain 
\begin{align*}
	|H_{2,1}(F_{f^{-1}}/2)|=\frac{1}{2304}.
\end{align*}

\noindent{\bf Case 2.} Let $\tau_{1}=0$. Then from \eqref{eq-4.5}, we see that 
\begin{align*}
	|H_{2,1}(F_{f^{-1}}/2)|=\frac{1}{144}|\tau_{2}|^2\leq\frac{1}{144}.
\end{align*}

\noindent{\bf Case 3.} Suppose that $\tau_{1}\in(0,1)$. Applying the triangle inequality in \eqref{eq-4.5} and by using the fact that $|\tau_{3}|\leq 1$, we obtain
\begin{align}\label{eq-4.6}
	\nonumber|H_{2,1}(F_{f^{-1}}/2)|&\leq\frac{1}{2304}\left(|-\tau^4_{1} -4(1-\tau^2_{1})\tau^2_{1}\tau_{2}-8(1-\tau^2_{1})(2+\tau^2_{1})\tau^2_{2}| \right.\\&\nonumber\left.\quad +24\tau_{1}(1-\tau^2_{1}) (1-|\tau^2_{2}|) \right)\\&=\frac{1}{96}\tau_{1}(1-\tau^2_{1}) \left(|A+B\tau_{2}+C\tau^2_{2}| +1 -|\tau_{2}|^2\right),
\end{align}
where 
\begin{align*}
	A:=\frac{-\tau^3_{1}}{24(1-\tau^2_{1})}, \;\; B:=-\frac{\tau_{1}}{6}
	\;\; \mbox{and}\;\; C:=\frac{-(2+\tau^2_{1})}{3\tau_{1}}.
\end{align*}
Observe that $AC>0$, so we can apply case (i) of Lemma \ref{lem-2.2}. Now we check all the conditions of case (i). \vspace{2mm}

\noindent{\bf 3(a)} We note that the inequality
\begin{align*}
	|B|-2(1-|C|)=\frac{\tau_{1}}{6} -2\left(1-\frac{(2+\tau^2_{1})} {3\tau_{1}}\right)=\frac{5\tau^2_{1} -12\tau_{1} +8}{6\tau_{1}}> 0
\end{align*}
is true for all $\tau_{1}\in(0,1)$. Thus it follows from  Lemma \ref{lem-2.2} and the inequality \eqref{eq-4.6} that
\begin{align*}
	|H_{2,1}(F_{f^{-1}}/2)|\leq\frac{1}{96}\tau_{1}(1-\tau^2_{1})\left(|A|+|B|+|C|\right)=\frac{1}{2304}\left(16-4\tau^2_{1}-11\tau^4_{1}\right)\leq\frac{1}{144}.
\end{align*}
\noindent{\bf 3(b)} Next, it is easy to check that
\begin{align*}
	|B|-2(1-|C|)=\frac{\tau_{1}}{6} -2\left(1-\frac{(2+\tau^2_{1})} {3\tau_{1}}\right)=\frac{5\tau^2_{1} -12\tau_{1} +8}{6\tau_{1}}< 0
\end{align*}
which is not true for all $\tau_{1}\in(0,1)$.\vspace{1.2mm}

Summarizing Cases $1$, $2$, and $3$, the inequality \eqref{eq-4.1} is established.\vspace{2mm}

To complete the proof, it is sufficient to show that the bound is sharp. In order to show that we consider the function $f\in\mathcal{A}$ as follows
\begin{align*}
    1+ \frac{zf^{\prime\prime}(z)}{f^{\prime}(z)}=\frac{1}{1 -z^2}, \;\; z\in\mathbb{D}, 
\end{align*}
with $a_2=a_4=0$ and $a_3=1/6$. By a simple computation, it can be easily shown that $|H_{2,1}(F_{f^{-1}}/2)|=1/144$. This completes the proof.
\end{proof}

\section{Sharp bound of $ |H_{2,1}(F_{f^{-1}}/2)| $ for the class $\mathcal{R}(1/2)$}
We obtain the following sharp bound of $ |H_{2,1}(F_{f^{-1}}/2)| $ for functions in the class $\mathcal{R}(1/2)$. 
\begin{thm}\label{th-5.1}
Let $ f\in\mathcal{R}(1/2) $. Then
\begin{align}\label{eq-5.1}
	|H_{2,1}(F_{f^{-1}}/2)|\leq\frac{1}{36}.
\end{align}
The inequality is sharp.
\end{thm}
\begin{proof}
Let $f\in\mathcal{R}(1/2)$. Then in view of \eqref{eq-1.1} and \eqref{eq-1.9}, it follows that
\begin{align}\label{eq-5.2}
	 f^{\prime}(z)=\frac{1}{2}(p(z) +1), \;\; z\in\mathbb{D}, 
\end{align} 
for some $p\in\mathcal{P}$ of form \eqref{eq-2.1}. A similar argument as used previously shows that $\tau_1\in[0,1]$. Substituting \eqref{eq-1.1} and \eqref{eq-2.1} into \eqref{eq-5.2} and equating the coefficients, we obtain
\begin{align}\label{eq-5.3}
	a_2=\frac{1}{4}c_1, \;\; a_3=\frac{1}{6}c_2 \;\;\mbox{and}\;\; a_4=\frac{1}{8}c_3.
\end{align}
Using \eqref{eq-1.6} and \eqref{eq-4.3}, it is easy to see that
\begin{align}\label{eq-5.4}
	H_{2,1}(F_{f^{-1}}/2)&=\frac{1}{48}\left(13a^4_2 -12a^2_2 a_3 - 12 a^2_3 + 12 a_2 a_4\right)\\&\nonumber=\frac{1}{36864}\left(39c^4_{1} -96c^2_{1} c_2 -256c^2_{2} +288c_1 c_3\right).
\end{align}
By the Lemma \ref{lem-2.1} and \eqref{eq-4.4}, a simple computation shows that
\begin{align}\label{eq-5.5}
	H_{2,1}(F_{f^{-1}}/2)&=\frac{1}{2304}\left(-\tau^4_{1} -32(1-\tau^2_{1})\tau^2_{1}\tau_{2}-8(1-\tau^2_{1})(8+\tau^2_{1})\tau^2_{2} \right.\\&\nonumber\left.\quad +72\tau_{1}\tau_{3}(1-\tau^2_{1})(1-|\tau^2_{2}|)\right).
\end{align}
Below, we discuss the following possible cases on $\tau_{1}$: \vspace{2mm}
	
\noindent{\bf Case 1.} Suppose that $\tau_{1}=1$. Then from \eqref{eq-5.5}, we easily obtain 
\begin{align*}
	|H_{2,1}(F_{f^{-1}}/2)|=\frac{1}{2304}.
\end{align*}
	
\noindent{\bf Case 2.} Let $\tau_{1}=0$. Then from \eqref{eq-5.5}, we see that 
\begin{align*}
	|H_{2,1}(F_{f^{-1}}/2)|=\frac{1}{36}|\tau_{2}|^2\leq\frac{1}{36}.
\end{align*}
	
\noindent{\bf Case 3.} Suppose that $\tau_{1}\in(0,1)$. Applying the triangle inequality in \eqref{eq-5.5} and by using the fact that $|\tau_{3}|\leq 1$, we obtain
\begin{align}\label{eq-5.6}
	\nonumber|H_{2,1}(F_{f^{-1}}/2)|&\leq\frac{1}{2304}\left(|-\tau^4_{1} -32(1-\tau^2_{1})\tau^2_{1}\tau_{2}-8(1-\tau^2_{1})(8+\tau^2_{1})\tau^2_{2}| \right.\\&\nonumber\left.\quad +72\tau_{1}(1-\tau^2_{1}) (1-|\tau^2_{2}|) \right)\\&=\frac{1}{32}\tau_{1}(1-\tau^2_{1}) \left(|A+B\tau_{2}+C\tau^2_{2}| +1 -|\tau_{2}|^2\right),
\end{align}
where 
\begin{align*}
	A:=\frac{-\tau^3_{1}}{72(1-\tau^2_{1})}, \;\; B:=-\frac{4\tau_{1}}{9}
		\;\; \mbox{and}\;\; C:=\frac{-(8+\tau^2_{1})}{9\tau_{1}}.
\end{align*}
Observe that $AC>0$, so we can apply case (i) of Lemma \ref{lem-2.2}. Now we check all the conditions of case (i). \vspace{2mm}
	
\noindent{\bf 3(a)} We note the inequality
\begin{align*}
	|B|-2(1-|C|)=\frac{4\tau_{1}}{9} -2\left(1-\frac{(8+\tau^2_{1})} {9\tau_{1}}\right)=\frac{6\tau^2_{1} -18\tau_{1} +16}{9\tau_{1}}> 0
\end{align*}
which is true for all $\tau_{1}\in(0,1)$. Thus it follows from  Lemma \ref{lem-2.2} and the inequality \eqref{eq-5.6} that
\begin{align*}
	|H_{2,1}(F_{f^{-1}}/2)|\leq\frac{1}{32}\tau_{1}(1-\tau^2_{1})\left(|A|+|B|+|C|\right)=\frac{1}{2304}\left(64-24\tau^2_{1}-39\tau^4_{1}\right)\leq\frac{1}{36}.
\end{align*}
\noindent{\bf 3(b)} Next, it is easy to check that
\begin{align*}
	|B|-2(1-|C|)=\frac{4\tau_{1}}{9} -2\left(1-\frac{(8+\tau^2_{1})} {9\tau_{1}}\right)=\frac{6\tau^2_{1} -18\tau_{1} +16}{9\tau_{1}}< 0
\end{align*}
which is not true for all $\tau_{1}\in(0,1)$.\vspace{1.2mm}
	
Summarizing Cases $1$, $2$, and $3$, the inequality \eqref{eq-5.1} is established.\vspace{2mm}
	
To complete the proof, it is sufficient to show that the bound is sharp. In order to show that we consider the function $f\in\mathcal{A}$ as follows
\begin{align*}
	f^{\prime}(z)=\frac{1}{1 -z^2}, \;\; z\in\mathbb{D}, 
\end{align*}
with $a_2=a_4=0$ and $a_3=1/3$. By a simple computation, it can be easily shown that $|H_{2,1}(F_{f^{-1}}/2)|=1/36$. This completes the proof.
\end{proof}
\section{Declarations}
\noindent\textbf{Compliance of Ethical Standards:}\\

\noindent\textbf{Conflict of interest.} The authors declare that there is no conflict  of interest regarding the publication of this paper.\vspace{1.5mm}

\noindent\textbf{Data availability statement.}  Data sharing is not applicable to this article as no datasets were generated or analyzed during the current study.

\end{document}